\documentclass[12pt]{amsart}
\usepackage[margin=1.2in]{geometry}
\usepackage{graphicx,latexsym}
\usepackage{tikz}
\usepackage{amsfonts, amssymb, amsmath, amsthm, bm, hyperref,verbatim}
\hypersetup{
	colorlinks=true
}
\usepackage{tikz}
\usepackage{tikz-cd}
\usetikzlibrary{matrix,arrows}
\usepackage{mathrsfs}
\allowdisplaybreaks

\newcommand{\N}{\mathbb{N}}
\newcommand{\ZZ}{\mathbb{Z}}
\newcommand{\R}{\mathbb{R}}

\renewcommand{\S}{\mathcal{S}}

\newcommand{\Z}{\mathcal{Z}}

\newcommand{\V}{\mathscr{V}}

\newcommand{\OO}{\mathcal{O}}
\newcommand{\OM}{\OO_{M}}

\newcommand{\beq}{\begin{eqnarray}}
\newcommand{\eeq}{\end{eqnarray}}

\newcommand{\beqs}{\begin{eqnarray*}}
\newcommand{\eeqs}{\end{eqnarray*}}

\newtheorem{theorem}{Theorem}[section]

\newtheorem{lemma}[theorem]{Lemma}

\theoremstyle{definition}
\newtheorem{definition}[theorem]{Definition}
\newtheorem{example}[theorem]{Example}

\theoremstyle{remark}

\numberwithin{equation}{section}

\usetikzlibrary{arrows,chains,matrix,positioning,scopes}
\makeatletter
\tikzset{join/.code=\tikzset{after node path={%
\ifx\tikzchainprevious\pgfutil@empty\else(\tikzchainprevious)%
edge[every join]#1(\tikzchaincurrent)\fi}}}
\makeatother
\tikzset{>=stealth',every on chain/.append style={join},
         every join/.style={->}}
\tikzstyle{labeled}=[execute at begin node=$\scriptstyle,
   execute at end node=$]
\newcommand\condM{\operatorname{M}}
\newcommand\condN{\operatorname{N}}

\newcommand\condDN{\operatorname{DN}}

\newcommand\condooOmega{\overline{\overline{\Omega}}}

\newcommand\ev[2]{\langle#1,#2\rangle}

\begin{document}
\title[Barrelled weighted $(PLB)$-spaces of ultradifferentiable functions]{A note on the barrelledness of weighted $(PLB)$-spaces of ultradifferentiable functions}

\author[A. Debrouwere]{Andreas Debrouwere}
\address{Department of Mathematics and Data Science \\ Vrije Universiteit Brussel, Belgium\\ Pleinlaan 2 \\ 1050 Brussels \\ Belgium}
\email{andreas.debrouwere@vub.be}

\author[L. Neyt]{Lenny Neyt}
\thanks{L. Neyt gratefully acknowledges support by FWO-Vlaanderen through the postdoctoral grant 12ZG921N}
\address{Department of Mathematics: Analysis, Logic and Discrete Mathematics\\ Ghent University\\ Krijgslaan 281\\ 9000 Gent\\ Belgium}
\email{lenny.neyt@UGent.be}


\begin{abstract}
In this note we consider weighted $(PLB)$-spaces of ultradifferentiable functions  defined via a weight function and a weight system, as introduced in our previous work  \cite{D-N-WeighPLBSpUltraDiffFuncMultSp}.  We provide a complete characterization of when these spaces are ultrabornological and barrelled in terms of  the defining weight system, thereby improving the  main Theorem 5.1 of  \cite{D-N-WeighPLBSpUltraDiffFuncMultSp}. 
In particular, we obtain that the multiplier space of the  Gelfand-Shilov space $\Sigma^{r}_{s}(\R^{d})$ of Beurling type is ultrabornological, whereas the one of the Gelfand-Shilov space $\S^{r}_{s}(\R^{d})$ of Roumieu type is not barrelled. 
\end{abstract}

\maketitle

\section{Introduction}

 $(PLB)$-spaces, i.e., countable projective limits of countable inductive limits of Banach spaces, are an important class of locally convex spaces in modern functional analysis. The problem when such spaces are barrelled and ultrabornological has been thoroughly  studied, both from an abstract point of view as well as for concrete function spaces; see the survey article  \cite{D-ClassicalPLSSp} and the references therein. This question is motivated by the fact that these conditions are necessary in order to be able to apply functional analytic tools such as the Banach-Steinhaus theorem, the  De Wilde open mapping and closed graph theorems or the theory of the derived projective limit functor \cite{W-DerivFunctorsFuncAnal}.   
 
In our recent article \cite{D-N-WeighPLBSpUltraDiffFuncMultSp}, we studied weighted $(PLB)$-spaces of ultradifferentiable functions defined via a weight function and a weight system. We characterized when such spaces  are ultrabornological and barrelled in terms of a condition on the defining weight system.  This may be seen as an analogue of Grothendieck's classical result that the $(PLB)$-space $\mathcal{O}_M$ of slowly increasing smooth functions is ultrabornological \cite{Grothendieck} in the context of ultradifferentiable functions. However, we were only able to prove that the  condition on the  weight system was necessary for the spaces to be barrelled under a rather unnatural  growth assumption on the weight function and the weight system (see below for details). The goal of the present article is to show the necessity of the condition without assuming this growth constraint.

Let $\omega : [0, \infty) \rightarrow [0, \infty)$ be a weight function (in the sense of \cite{B-M-T-UltradiffFuncFourierAnal}, cf. Section \ref{sec:WeightFunctions}) and set $\phi(x) = \omega(e^{x})$. We denote by $\phi^{*}(y) = \sup_{x \geq 0} \{xy - \phi(x) \}$ the Young conjugate of $\phi$. Let $\V = \{ v_{\lambda} \mid \lambda \in (0, \infty) \}$ be a family of continuous functions $v_{\lambda} : \R^{d} \rightarrow (0, \infty)$ such that $1 \leq v_{\lambda} \leq v_{\mu}$ for all $\mu \leq \lambda$ ($\V$  is called a weight system \cite{D-N-V-NuclGSSpKernThm}, cf. Section \ref{sec:WeightSystems}). We are interested in  the following weighted $(PLB)$-spaces of ultradifferentiable functions of Beurling and Roumieu type
	\begin{equation}
		\label{eq:ZSpaces}
		\Z^{(\omega)}_{(\V)} := \varprojlim_{h \rightarrow 0^{+}} \varinjlim_{\lambda \rightarrow 0^{+}} \Z^{\omega, h}_{v_{\lambda}} , \qquad \Z^{\{\omega\}}_{\{\V\}} := \varprojlim_{\lambda \rightarrow \infty} \varinjlim_{h \rightarrow \infty} \Z^{\omega, h}_{v_{\lambda}} , 
	\end{equation}
where $\Z^{\omega, h}_{v_{\lambda}}$ denotes the Banach space consisting of all $\varphi \in C^{\infty}(\R^{d})$ such that
	\[ \|\varphi\|_{\Z^{\omega, h}_{v_{\lambda}}} := \sup_{\alpha \in \N^{d}} \sup_{x \in \R^{d}} \frac{|\varphi^{(\alpha)}(x)|}{v_{\lambda}(x)} \exp\left(-\frac{1}{h}\phi^{*}(h|\alpha|)\right) < \infty .  \]
We use $\Z^{[\omega]}_{[\V]}$ as a common notation for $\Z^{(\omega)}_{(\V)}$ and $\Z^{\{\omega\}}_{\{\V\}}$; a similar convention will be used for other notations as well. 
We  also need the corresponding Gelfand-Shilov spaces \cite{D-N-V-NuclGSSpKernThm}, which are defined as
	\[ \S^{(\omega)}_{(\V)} := \varprojlim_{h \rightarrow 0^{+}} \S^{\omega, h}_{v_{h}} , \qquad \S^{\{\omega\}}_{\{\V\}} := \varinjlim_{h \rightarrow \infty} \S^{\omega, h}_{v_{h}} , \]
where $\S^{\omega, h}_{v_{h}}$ denotes the Banach space consisting of all $\varphi \in C^{\infty}(\R^{d})$ such that
	\[ \|\varphi\|_{\S^{\omega, h}_{v_{h}}} := \sup_{\alpha \in \N^{d}} \sup_{x \in \R^{d}} |\varphi^{(\alpha)}(x)| v_{h}(x) \exp\left(-\frac{1}{h}\phi^{*}(h|\alpha|)\right) < \infty . \]
Lastly, we introduce the conditions on a weight system that characterize when $\Z^{[\omega]}_{[\V]}$ is ultrabornological and barrelled.

\begin{definition}\label{def-intro} Let $\V=  \{ v_{\lambda} \mid \lambda \in (0, \infty) \}$ be a weight system.
\begin{itemize}
\item[$(i)$] $\V$ is said to satisfy the condition $(\condDN)$ if
			\[ \exists \lambda > 0 ~ \forall \mu \leq \lambda, \theta \in (0, 1) ~ \exists \nu \leq \mu, C > 0 ~ \forall x \in \R^{d}~ : ~ v_{\mu}(x) \leq C v^{\theta}_{\lambda}(x) v^{1 - \theta}_{\nu}(x). \]

\item[$(ii)$] $\V$ is said to satisfy the condition $(\condooOmega)$ if
			\[ \forall \lambda > 0 ~ \exists \mu \geq \lambda ~ \forall \nu \geq \mu, \theta \in (0, 1) ~ \exists C > 0 ~ \forall x \in \R^{d} ~: ~ v_{\mu}(x) \leq C v^{\theta}_{\lambda}(x) v^{1 - \theta}_{\nu}(x) . \]
			\end{itemize}
	\end{definition}	

Our main result may now be stated as follows; we refer to Section \ref{sec:Preliminaries} for the definition of the  conditions $[\condM]$ and $[\condN]$.

\begin{theorem}
		\label{t:main}
		Let $\omega$ be a weight function and let $\V$ be a weight system satisfying $[\condM]$ and $[\condN]$. Suppose that $\S^{[\omega]}_{[\V]} \neq \{0\}$. Then, the following statements are equivalent:
			\begin{itemize}
				\item[$(i)$] $\V$ satisfies $(\condDN)$ (Beurling case); $\V$ satisfies $(\condooOmega)$ (Roumieu case).
				\item[$(ii)$] $\Z^{[\omega]}_{[\V]}$ is ultrabornological.
				\item[$(iii)$] $\Z^{[\omega]}_{[\V]}$ is barrelled.
			\end{itemize}
	\end{theorem}
	
	The assumption that $\mathcal{S}^{[\omega]}_{[\V]}$ is non-trivial in Theorem \ref{t:main} should be interpreted as an implicit growth condition on $\omega$ and $\V$ under which this result is valid. 
We have that $\mathcal{S}^{[\omega]}_{[\V]} \neq \{ 0\}$ if either $\omega$ is non-quasianalytic or if  the following condition is satisfied
\begin{itemize} 
\item[$(NT)$]  $\omega(t) = O(t^{1/r})$ and $\forall \lambda ~   \exists h ~  v_\lambda(x) = O(e^{h|x|^{1/s}})$ for some $r,s >0$ with $r+s > 1$ (Beurling case);   $\omega(t) = O(t^{1/r})$ and $\forall h ~  
~   \exists \lambda ~  v_\lambda(x) = O(e^{h|x|^{1/s}})$ for some $r,s >0$ with $r+s \geq 1$ (Roumieu case).
\end{itemize}
The sufficiency of the condition $(NT)$ follows from the fact that the classical Gelfand-Shilov space  $\S^{r}_{s}$ is non-trivial if and only if $r+s \geq 1$  \cite[Section 8]{G-S-GenFunc2}. We refer to \cite{D-V2018} for conditions on $\V$ ensuring that  $\mathcal{S}^{[\omega]}_{[\V]} \neq \{ 0\}$ in the special case that $\omega(t) = t$. In general, the characterization of the non-triviality of $\mathcal{S}^{[\omega]}_{[\V]}$ is an open problem.

 We now compare Theorem \ref{t:main} with our previous work \cite{D-N-WeighPLBSpUltraDiffFuncMultSp}.  The implication $(i) \Rightarrow (ii)$ has been shown in  \cite[Theorem 5.1]{D-N-WeighPLBSpUltraDiffFuncMultSp}, whereas $(ii) \Rightarrow (iii)$ holds for any locally convex space. However, in \cite[Theorem 5.1]{D-N-WeighPLBSpUltraDiffFuncMultSp} we could only show $(iii) \Rightarrow (i)$ under the additional assumption that $\S^{[\omega]}_{[\V]}$ is so-called \emph{Gabor accessible}. The latter means that  $\S^{[\omega]}_{[\V]}$ contains a pair of dual windows generating a Gabor frame; see  \cite[Section 4]{D-N-WeighPLBSpUltraDiffFuncMultSp} for more information. The existence of such a frame enabled us to prove the implication  $(iii) \Rightarrow (i)$ by reducing it to the corresponding problem for certain weighted $(PLB)$-sequence spaces and then using results from \cite{A-B-B-ProjLimWeighLBSpContFunc} about the barrelledness of such spaces.  Similarly as above, one should interpret Gabor accessibility as an implicit growth condition on $\omega$ and $\V$. Classical results from time-frequency analysis concerning window design \cite{C-PairsDualGaborFrameGenCompSupp, Janssen} (see also see  \cite[Proposition 4.5]{D-N-WeighPLBSpUltraDiffFuncMultSp}) imply that  $\S^{[\omega]}_{[\V]}$ is Gabor accessible if either $\omega$ is non-quasianalytic or if the following condition is satisfied
\begin{itemize} 
\item[$(GA)$]$\omega(t) = o(t^2)$ and $\forall \lambda ~ \forall h ~: ~ v_\lambda(x) = O(e^{h|x|^2})$ (Beurling case); $\omega(t) = O(t^2)$ and $\forall h ~\exists \lambda ~: ~ v_\lambda(x) = O(e^{h|x|^2})$ (Roumieu case). 
\end{itemize}
It seems that these are the best known conditions on $\omega$ and $\V$ ensuring that $\S^{[\omega]}_{[\V]}$ is Gabor accessible.  In particular, it is an open problem whether every non-trivial Gelfand-Shilov space is Gabor accessible. Note that the condition $(GA)$ is much more restrictive than $(NT)$. The present article is devoted to the proof of $(iii) \Rightarrow (i)$ in Theorem \ref{t:main} under the sole assumption that $\S^{[\omega]}_{[\V]} \neq \{0\}$. We develop a completely novel method to achieve this goal. Namely, we  combine an abstract result from the theory of projective spectra of $(LB)$-spaces \cite{V-LecturesProjSpecLBSp,W-DerivFunctorsFuncAnal} with the mapping properties of the short-time Fourier transform  \cite{G-FoundTFAnal} on various spaces related to $\Z^{[\omega]}_{[\V]}$.

One of the main motivations to study the spaces $\Z^{[\omega]}_{[\V]}$ is that they are the multiplier spaces of the Gelfand-Shilov spaces $\S^{[\omega]}_{[\V]}$. More precisely, consider the abstract multiplier space
	\[ \OM(\S^{[\omega]}_{[\V]}) = \{ f \in \S^{\prime [\omega]}_{[\V]} \mid \varphi \cdot f \in \S^{[\omega]}_{[\V]} \text{ for all } \varphi \in \S^{[\omega]}_{[\V]} \}  \]
endowed with the locally convex topology induced by the embedding
	\[ \OM(\S^{[\omega]}_{[\V]}) \rightarrow L_{b}(\S^{[\omega]}_{[\V]}, \S^{[\omega]}_{[\V]}) , \quad f \mapsto (\varphi \mapsto \varphi \cdot f) . \]
 Under the conditions of Theorem \ref{t:main} and the additional natural assumption
	\begin{equation} 
		\label{eq:CondS}
		\forall \lambda, \mu ~ \exists \nu \leq \lambda, \mu ~ (\forall \nu ~ \exists \lambda, \mu \geq \nu ) : \quad \sup_{x \in \R^{d}} \frac{v_{\lambda}(x) v_{\mu}(x)}{v_{\nu}(x)} < \infty , 
	\end{equation}
we have shown in \cite[Theorem 5.3]{D-N-WeighPLBSpUltraDiffFuncMultSp} that $\OM(\S^{[\omega]}_{[\V]}) = \Z^{[\omega]}_{[\V]}$
as locally convex spaces.  Consequently,  Theorem \ref{t:main} characterizes when these multiplier spaces are ultrabornological and barrelled. In this context, the assumption $\S^{[\omega]}_{[\V]} \neq \{0\}$ in Theorem \ref{t:main}  is very natural. Let us explicitly state this result in a specific case. Let $\eta$ be another weight function and consider the weight system $\V_{\eta} = \{ e^{\frac{1}{\lambda}\eta} \mid \lambda	\in (0, \infty) \}$. Set $\S^{[\omega]}_{[\eta]} = \S^{[\omega]}_{[\V_{\eta}]}$. We find the following result about the locally convex structure of $\OM(\S^{[\omega]}_{[\eta]})$.
	
	\begin{theorem}
		\label{t:GS}
		Let $\omega$ and $\eta$ be two weight functions. Suppose that $\S^{[\omega]}_{[\eta]} \neq \{0\}$. Then,
			\begin{itemize}
				\item[$(i)$] $\OM(\S^{(\omega)}_{(\eta)})$ is ultrabornological.
				\item[$(ii)$] $\OM(\S^{\{\omega\}}_{\{\eta\}})$ is not barrelled.
			\end{itemize}
	\end{theorem}
In \cite{D-N-WeighPLBSpUltraDiffFuncMultSp}, we were only able to show Theorem \ref{t:GS}$(ii)$ under the additional assumption that either $\omega$ or $\eta$ is non-quasianalytic or that $\omega(t) = O(t^2)$ and  	$\eta(t) = O(t^2)$ (cf.\ the above discussion). 
Note that Theorem \ref{t:GS} particularly applies to the classical Gelfand-Shilov spaces \cite{G-S-GenFunc2} 
	\[ \Sigma^{r}_{s} = \S^{(t^{1/r})}_{(t^{1/s})} , \qquad \S^{r}_{s} = \S^{\{t^{1/r}\}}_{\{t^{1/s}\}}, \qquad r,s >0. \]
	
	

\section{Preliminaries}
\label{sec:Preliminaries}

We introduce several notions and results  that will be used throughout this paper. 

\subsection{Weight functions}
\label{sec:WeightFunctions}

A non-decreasing continuous function $\omega : [0, \infty) \rightarrow [0, \infty)$ is called a \emph{weight function} (in the sense of Braun, Meise, and Taylor \cite{B-M-T-UltradiffFuncFourierAnal}) if $\omega(0) = 0$ and $\omega$ satisfies the following properties:
	\begin{itemize}
		\item[$(\alpha)$] $\omega(2t) = O(\omega(t))$ as $t \rightarrow \infty$;
		\item[$(\gamma)$] $\log t = o(\omega(t))$ as $t \rightarrow \infty$;
		\item[$(\delta)$] $\phi : [0, \infty) \rightarrow [0, \infty)$, $\phi(x) = \omega(e^{x})$, is convex.
	\end{itemize}
We extend $\omega$ to $\R^{d}$ as the radial function $\omega(x) = \omega(|x|)$, $x \in \R^{d}$. Condition $(\alpha)$ implies that there is a constant $S > 0$ such that \cite[Lemma 1]{B-M-T-UltradiffFuncFourierAnal}
	\begin{equation}
		\label{eq:WeighFuncModerate}
		\omega(x + y) \leq S(\omega(x) + \omega(y) + 1) , \qquad x, y \in \R^{d} . 
	\end{equation}
The \emph{Young conjugate of $\phi$} is defined as
	\[ \phi^{*} : [0, \infty) \rightarrow [0, \infty), \quad \phi^{*}(y) = \sup_{x \geq 0}\{ xy - \omega(x)\}.  \]
The function $\phi^{*}$ is convex and increasing, $(\phi^{*})^{*} = \phi$, and the function $y \mapsto \phi^{*}(y) / y$ is increasing on $[0, \infty)$ and tends to infinity as $y \rightarrow \infty$. 

\begin{lemma}[{\cite[Lemma 2.6]{Heinrich}\label{M12}}]
Let $\omega$ be a weight function. Then,
\begin{itemize}
\item[$(i)$] For all $h,k,l > 0$ there are $m,C > 0$ such that
\begin{equation}
 \frac{1}{m}\phi^*(m(y+l)) + ky \leq \frac{1}{h}\phi^*(hy) + \log C, \qquad  y \geq 0.
\label{ineq-0}
\end{equation}
\item[$(ii)$] For all $m,k,l > 0$ there are $h,C > 0$ such that \eqref{ineq-0} holds.
\end{itemize}
\end{lemma}

\subsection{Weight systems}
\label{sec:WeightSystems}

A family $\V = \{ v_{\lambda} \mid \lambda \in (0, \infty) \}$ of continuous functions $v_{\lambda} : \R^{d} \rightarrow (0, \infty)$ is called a \emph{weight system} \cite{D-N-V-NuclGSSpKernThm} if $1 \leq v_{\lambda}(x) \leq v_{\mu}(x)$ for all $x \in \R^{d}$ and $\mu \leq \lambda$. We write $\widetilde{v}(x) = v(-x)$ for reflection about the origin. Given a weight system $\V$, we set $\widetilde{\V} = \{ \widetilde{v}_{\lambda} \mid \lambda \in (0, \infty) \}$. 

We will employ the following  conditions on a weight system  $\V$:
	\begin{itemize}
		\item[$(\condM)$] $\forall \lambda > 0 ~ \exists \mu, \nu \leq \lambda ~ \exists C > 0 ~ \forall x, y \in \R^{d} : v_{\lambda}(x + y) \leq Cv_{\mu}(x) v_{\nu}(y)$;
		\item[$\{\condM\}$] $\forall \mu, \nu > 0 ~ \exists \lambda \geq \mu, \nu ~ \exists C > 0 ~ \forall x, y \in \R^{d} : v_{\lambda}(x + y) \leq Cv_{\mu}(x) v_{\nu}(y)$;
		\item[$(\condN)$] $\forall \lambda > 0 ~ \exists \mu \leq \lambda : v_{\lambda} / v_{\mu} \in L^{1}$;
		\item[$\{\condN\}$] $\forall \mu > 0 ~ \exists \lambda \geq \mu : v_{\lambda} / v_{\mu} \in L^{1}$.
	\end{itemize}
We will use $[\condM]$ as a common notation for $(\condM)$ and $\{\condM\}$; a similar convention will be used for other notations. In addition, we often first state assertions for the Beurling case followed in parenthesis by the corresponding ones for the Roumieu case.

Note that $[\condM]$ and $[\condN]$ imply that \cite[Lemma 3.1]{D-N-V-NuclGSSpKernThm}:
	\begin{equation}
		\label{eq:improvedN}
		\forall \lambda > 0 ~ \exists \mu \leq \lambda ~ ( \forall \mu > 0 ~ \exists \lambda \geq \mu ) : \quad v_{\lambda} / v_{\mu} \in L^{1} \cap C_{0} , 
	\end{equation}
where $C_{0}$ denotes the space of continuous functions on $\R^{d}$ vanishing at infinity.

	\begin{example}
		For a given weight function $\omega$, we define the weight system
			\[ \V_{\omega} = \{ e^{\frac{1}{\lambda} \omega} \mid \lambda \in (0, \infty) \} . \]
		Then, $\V_{\omega}$ satisfies $[\condM]$, $[\condN]$ and \eqref{eq:CondS}. Furthermore, it satisfies $(\condDN)$ but not $(\condooOmega)$ \cite[Lemma 3.5]{D-N-WeighPLBSpUltraDiffFuncMultSp}. 
	\end{example}
	\noindent In view of the above example, Theorem \ref{t:GS} is a direct consequence of Theorem \ref{t:main}. 
	
	The conditions $(\condDN)$ and $(\condooOmega)$  (see Definition \ref{def-intro}) may be characterized in terms of a $(wQ)$-type condition \cite{A-B-B-ProjLimWeighLBSpContFunc}:
	\begin{lemma}[{\cite[Lemma 6.4]{D-N-WeighPLBSpUltraDiffFuncMultSp}}]
\label{wQ}
	Let $\omega$ be a weight function and let $\V$ be a weight system.  \\
	$(i)$ $\V$ satisfies $(\condDN)$  if and only if
	\begin{gather*}
			\forall N \in \N ~ \exists M \geq N, n \in \N ~ \forall K \geq M, m \in \N ~ \exists k \in \N, C > 0~ \forall (x, \xi) \in \R^{2d}~:\\
				v_{1/m}(x) e^{-M \omega(\xi)} \leq C \left( v_{1/n}(x) e^{-N \omega(\xi)} + v_{1/k}(x) e^{-K \omega(\xi)}  \right).
			\end{gather*}

	$(ii)$ $\V$ satisfies $(\condooOmega)$  if and only if 
	\begin{gather*}
				\forall N \in \N ~ \exists M \geq N, n \in \N ~ \forall K \geq M, m \in \N ~ \exists k \in \N, \exists C > 0~ \forall (x, \xi) \in \R^{2d}~:\\
				v_{M}(x) e^{-\frac{1}{m} \omega(\xi)} \leq C \left( v_{N}(x) e^{-\frac{1}{n} \omega(\xi)} + v_{K}(x) e^{-\frac{1}{k} \omega(\xi)}\right).
			\end{gather*}

	\end{lemma}
	
\subsection{Weighted spaces of ultradifferentiable functions and the short-time Fourier transform}
\label{sec:GSSp}

Let $\omega$ be a weight function. For $h > 0$ and $v : \R^{d} \rightarrow (0, \infty)$ continuous, we define $\S^{\omega, h}_{v}$ as the Banach space consisting of all $\varphi \in C^{\infty}(\R^{d})$ such that
	\[ \|\varphi\|_{\S^{\omega, h}_{v}} := \sup_{\alpha \in \N^{d}} \sup_{x \in \R^{d}} |\varphi^{(\alpha)}(x)| v(x) \exp\left( - \frac{1}{h} \phi^{*}(h |\alpha|) \right) < \infty .  \]
Let $\V$ be a weight system. As in the introduction, we set 
	\[ \S^{(\omega)}_{(\V)} = \varprojlim_{h \rightarrow 0^{+}} \S^{\omega, h}_{v_{h}} , \qquad \S^{\{\omega\}}_{\{\V\}} = \varinjlim_{h \rightarrow \infty} \S^{\omega, h}_{v_{h}} . \]
Then, $\S^{(\omega)}_{(\V)}$ is a Fr\'{e}chet space and $\S^{\{\omega\}}_{\{\V\}}$ is a Hausdorff $(LB)$-space.  The  $(PLB)$-spaces  introduced in  \eqref{eq:ZSpaces} may be written as
	$$
		\Z^{(\omega)}_{(\V)} = \varprojlim_{h \rightarrow 0^{+}} \varinjlim_{\lambda \rightarrow 0^{+}} \S^{\omega, h}_{1/v_{\lambda}} , \qquad \Z^{\{\omega\}}_{\{\V\}} = \varprojlim_{\lambda \rightarrow \infty} \varinjlim_{h \rightarrow \infty} \Z^{\omega, h}_{1/v_{\lambda}}. 
	$$
 
 Next, we introduce the short-time Fourier transform;  we refer the reader to the book \cite{G-FoundTFAnal} for more information on this topic. 
  The translation and modulation operators are denoted by $T_{x} f(t) = f(t - x)$ and $M_{\xi} f(t) = e^{2 \pi i \xi \cdot t} f(t)$ for $x, \xi \in \R^{d}$. The \emph{short-time Fourier transform (STFT)} of $f \in L^{2}(\R^{d})$ with respect to a window  $\psi \in L^{2}(\R^{d})$ is defined as 
	\[ V_{\psi} f(x, \xi) = (f, M_{\xi} T_{x} \psi)_{L^{2}} = \int_{\R^{d}} f(t) \overline{\psi(t - x)} e^{- 2 \pi i \xi \cdot t} dt , \qquad (x, \xi) \in \R^{2d} .  \]
The following two structural properties of the STFT shall play a crucial role in this article:
\begin{equation}
\label{Plancherel}
(V_\psi f_1,V_\psi f_2)_{L^2(\R^{2d})} = \|\psi\|^2_{L^2}(f_1, f_2)_{L^2}, \qquad f_1, f_2 \in L^{2}(\R^{d}),
\end{equation}
and 
\begin{equation}
\label{repres}
|V_{\psi}(M_{\xi} T_{x} f)(y, \eta)| = |V_{\psi}(f)(y-x, \eta-\xi)|, \qquad  (x, \xi), (y,\eta) \in \R^{2d}.
\end{equation}

We now recall a result from \cite{D-N-WeighPLBSpUltraDiffFuncMultSp} concerning the mapping properties of the STFT on weighted spaces of ultradifferentiable functions. To this end, we introduce a class of weighted spaces of continuous functions on the time-frequency space. Given $v, w : \R^{d} \rightarrow (0, \infty)$ continuous, we write $C_{v \otimes w}(\R^{2d}_{x, \xi})$ for the Banach space consisting of all $\Phi \in C(\R^{2d})$ such that
	\[ \|\Phi\|_{v \otimes w} = \sup_{(x, \xi) \in \R^{2d}} |\Phi(x, \xi)| v(x) w(\xi) < \infty . \]

	\begin{lemma}[{\cite[Lemma 7.1]{D-N-WeighPLBSpUltraDiffFuncMultSp}}]
		\label{l:STFTGS} 
		Let $\omega$ be a weight function. Let $v_{i}: \R^{d} \to (0, \infty)$, $i = 1,2,3,4$, be continuous functions such that
			\begin{equation}
			\label{v14}
			v_{2}(x + t) \leq C_{0} v_{1}(x) \widetilde{v}_{4}(t), \qquad x, t \in \R^{d},
			\end{equation}
		for some $C_{0} > 0$ and $v_{4} / v_{3} \in L^{1}$.  Let $h_{i} > 0$, $i= 1,2$, be such that
			\[ \frac{1}{h_{1}} \phi^{*}(h_{1}(y + 1)) + (\log\sqrt{d}) y \leq \frac{1}{h_{2}} \phi^{*}(h_{2} y)  + \log C_{1}, \qquad y \geq 0, \]
		for some $C_{1} > 0$.  Let $\psi \in \S^{\omega, h_{1}}_{v_{3}}$. Then, the mapping 
			\[ V_{\psi} : \S^{\omega, h_{1}}_{v_{1}} \rightarrow C_{v_{2} \otimes e^{\frac{1}{h_{2}} \omega}}(\R^{2d}_{x,\xi}) \]
		is well-defined and continuous.
	\end{lemma}
	Let $v : \R^{d} \rightarrow (0, \infty)$ be continuous. We denote by $C_{v}$  the Banach space consisting of all $f \in C(\R^{d})$ such that
	\[ \|f\|_{v} = \sup_{x \in \R^{d}} |f(x)| v(x) < \infty . \]
	\begin{lemma}
		\label{l:STFTC} 
		Let $v_{i}: \R^{d} \to (0, \infty)$, $i = 1,2,3,4$, be continuous functions such that \eqref{v14} holds
		for some $C_{0} > 0$ and $v_{4} / v_{3} \in L^{1}$. Let $\psi \in C_{v_{3}}$. Then, the mapping 
			\[ V_{\psi} : C_{v_{1}} \rightarrow C_{v_{2} \otimes 1}(\R^{2d}_{x,\xi}) \]
		is well-defined and continuous.
	\end{lemma}
\begin{proof}
The proof is straightforward and therefore left to the reader.
\end{proof}
	
	We shall also need the following result from  \cite{D-N-WeighPLBSpUltraDiffFuncMultSp}.

	\begin{lemma}[{\cite[Lemma 7.2]{D-N-WeighPLBSpUltraDiffFuncMultSp}}]
		\label{l:TFSGS}
			Let $\omega$ be a weight function. Choose $C_0,A > 0$ such that
\begin{equation}
\omega(2\pi t) \leq A \omega(t) + \log C_0, \qquad t \geq 0.
\label{eq:ConstantL}
\end{equation}
Let $v_i: \R^d \to (0,\infty)$, $i = 1,2,3$, be continuous functions such that
\begin{equation}
v_2(x+t) \leq C_1v_1(x) v_3(t), \qquad x,t \in \R^d,
\label{in-2}
\end{equation}
for some $C_1 > 0$. Let $h_i > 0$, $i= 1,2$, be such that
\begin{equation}
\frac{1}{h_1} \phi^* (h_1y) + (\log 2) y \leq \frac{1}{h_2} \phi^* (h_2y) 
+ \log C_2, \qquad y \geq 0, 
\label{in-3}
\end{equation}
for some $C_2 > 0$. 
Then, there is a $C > 0$ such that
$$
\| M_\xi T_x \psi \|_{\mathcal{S}^{\omega,h_2}_{v_2}} \leq C \|  \psi \|_{\mathcal{S}^{\omega,h_1}_{v_1}} v_3(x) e^{\frac{A}{h_1}\omega(\xi)}, \qquad (x,\xi) \in \R^{2d},
$$
for all $\psi \in \mathcal{S}^{\omega,h_1}_{v_1}$.
		\end{lemma}

\section{Barrelled $(PLB)$-spaces}
\label{sec:ProjectiveSpectra}

In this short section, we explain the abstract result concerning barrelled $(PLB)$-spaces that will be used in the proof of the implication $(iii) \Rightarrow (i)$ in Theorem \ref{t:main}. 

For us, a \emph{projective spectrum} is a decreasing sequence $(X_{N})_{N \in \N}$ of locally convex spaces such that the inclusion mapping  $X_{N+1} \to  X_{N}$ is continuous for each $N \in \N$. Set 
$X=\varprojlim_{N \in \N} X_{N}$. The spectrum $(X_{N})_{N \in \N}$ is called \emph{strongly reduced} if
			\[ \forall N \in \N ~ \exists M \geq N ~: ~ X_{M} \subseteq \overline{X}^{X_{N}} . \]

	\begin{theorem}[{\cite[Theorem 3.3.6]{W-DerivFunctorsFuncAnal}}]
		\label{t:BarrelledImpliesP*2}
		Let $(X_{N})_{N \in \N}$ be a strongly reduced projective spectrum of Hausdorff $(LB)$-spaces $X_{N} = \varinjlim_{n \in \N} X_{N, n}$.  If $\varprojlim_{N \in \N} X_{N} $ is barrelled, then
			\begin{gather*}
				\forall N \in \N ~ \exists M \geq N, n \in \N ~ \forall K \geq M, m \in \N ~ \exists k \in \N, C > 0~\forall y \in X^{\prime}_{N} ~ :\\
				 \|y \|^{*}_{X_{M, m}} \leq C \left( \| y \|^{*}_{X_{N, n}} + \| y \|^{*}_{X_{K, k}} \right) ,
			\end{gather*}
		where $\|\cdot\|^{*}_{X_{L, l}}$ denotes the dual norm of $\| \cdot \|_{X_{L, l}}$,  $L, l \in \N$.
	\end{theorem}

\section{Proof of Theorem \ref{t:main}}
\label{sec:ProofMainThm}

The goal of this section is to prove Theorem \ref{t:main}. As already stated in the introduction, we only have to show the implication $(iii) \Rightarrow (i)$. We fix a weight function $\omega$ and a weight system $\V$ satisfying $[M]$ and $[N]$ such that $\S^{[\omega]}_{[\V]} \neq \{ 0\}$. Furthermore, we choose  $\psi \in \S^{[\omega]}_{[\V]} \cap \S^{[\omega]}_{[\widetilde{\V}]}$ with $\| \psi\|_{L^{2}} = 1$; the  existence of such a function is guaranteed by \cite[Lemma 7.5]{D-N-WeighPLBSpUltraDiffFuncMultSp}.  Let $S,A \geq 1$ be such that \eqref{eq:WeighFuncModerate} and \eqref{eq:ConstantL} hold.

 We now introduce two sequences $(h_{n})_{n \in \N}$ and $(\lambda_{n})_{n \in \N}$ of positive real numbers, whose definition depends on whether we are in the Beurling or in the Roumieu case: \\
 \emph{Beurling case}: Set $h_{0} = \lambda_{0} = 1$. By Lemma \ref{M12}$(i)$, we can inductively choose $h_{n+1} > 0$ such that $h_{n + 1} < h_n/\max\{A,S\}$ and
			\[ \frac{1}{h_{n + 1}} \phi^{*}\left(h_{n + 1} (y + 1) \right) + \log (\max\{2, \sqrt{d} \} y) \leq \frac{1}{h_{n}} \phi^{*}(h_{n} y) + \log C , \qquad y \geq 0 , \]
		for some $C = C_n > 0$. By \eqref{eq:improvedN} and $(M)$, we  can inductively choose $\lambda_{n+1} > 0$ such that $\lambda_{n + 1} < \lambda_n$,
			\[ \frac{v_{\lambda_{n}}}{v_{\lambda_{n + 1}}} \in L^{1} \cap C_{0} , \]
		and
			\[ v_{\lambda_{n}}(x + t) \leq C' v_{\lambda_{n + 1}}(x) v_{\lambda_{n + 1}}(t) , \qquad \forall x, t \in \R^{d} , \]
		for some $C' = C'_n > 0$. We may assume without loss of generality that $h_n \searrow 0$ and $\lambda_n \searrow 0$ as $n \to \infty$. \\
 \emph{Roumieu case}: Pick $h_{0}, \lambda_{0} > 0$ such that $\psi \in \S^{\omega, h_{0}}_{v_{\lambda_{0}}} \cap \S^{\omega, h_{0}}_{\widetilde{v}_{\lambda_{0}}}$. By Lemma \ref{M12}$(ii)$, we can inductively choose $h_{n+1} > 0$ such that $h_{n + 1} > \max\{A,S\} h_{n}$ and
			\[ \frac{1}{h_{n}} \phi^{*}\left(h_{n} (y + 1) \right) + \log (\max\{2, \sqrt{d}\} y) \leq \frac{1}{h_{n + 1}} \phi^{*}(h_{n + 1} y) + \log C , \qquad y \geq 0 , \]
		for some $C = C_n > 0$. By \eqref{eq:improvedN} and $\{M\}$, we  can inductively choose $\lambda_{n+1} > 0$ such that $\lambda_{n + 1} > \lambda_{n}$,
			\[ \frac{v_{\lambda_{n + 1}}}{v_{\lambda_{n}}} \in L^{1} \cap C_{0} , \]
		and
			\[ v_{\lambda_{n + 1}}(x + t) \leq C' v_{\lambda_{n}}(x) v_{\lambda_{n}}(t) , \qquad \forall x, t \in \R^{d} , \]
		for some $C' = C_ n > 0$. We may assume without loss of generality that $h_n \nearrow \infty$ and $\lambda_n \nearrow \infty$ as $n \to \infty$. 

For each $N \in \N$ we define the Hausdorff $(LB)$-spaces
	\[ \Z^{\omega, h_{N}}_{(\V)} = \varinjlim_{n \in \N} \S^{\omega, h_{N}}_{1/v_{\lambda_{n}}} , \qquad \Z^{\{\omega\}}_{v_{\lambda_{N}}} = \varinjlim_{n \in \N} \S^{\omega, h_{n}}_{1/v_{\lambda_{N}}} . \]
Note that
		$$
		\Z^{(\omega)}_{(\V)} = \varprojlim_{N \in \N} \Z^{\omega, h_{N}}_{(\V)} , \qquad \Z^{\{\omega\}}_{\{\V\}} = \varprojlim_{N \in \N} \Z^{\{\omega\}}_{v_{\lambda_{N}}}.
		$$ 

We will apply Theorem \ref{t:BarrelledImpliesP*2} to the projective spectra $( \Z^{\omega, h_{N}}_{(\V)})_{N \in \N}$ and $(\Z^{\{\omega\}}_{v_{\lambda_{N}}})_{N \in \N}$ to show the implication $(iii) \Rightarrow (i)$ in Theorem \ref{t:main}. Therefore, we start by showing that these spectra are strongly reduced. 	
	\begin{lemma}
		\label{l:ZStronglyReduced}
		The projective spectra $( \Z^{\omega, h_{N}}_{(\V)})_{N \in \N}$ and $(\Z^{\{\omega\}}_{v_{\lambda_{N}}})_{N \in \N}$ are strongly reduced.
	\end{lemma}
	
	\begin{proof}
 \emph{Beurling case}: It suffices to show that for all $n,N \in \N$ 
	$$
	\S^{\omega,h_{N+1}}_{1/v_{\lambda_{n}}} \subseteq \overline{\Z^{(\omega)}_{(\V)}}^{\S^{\omega,h_{N}}_{1/v_{\lambda_{n+1}}}}.
	$$
	Let $\varphi \in \S^{\omega,h_{N+1}}_{1/v_{\lambda_{n}}}$ be arbitrary. Choose $\chi \in \S^{(\omega)}_{(\V)}$ such that $\int_{\R^d}\chi(x)dx = 1$. Pick $\eta \in \mathcal{D}(\R^d)$ such that $\eta(0) = 1$ and set $\eta_j(x) = \eta(x/j)$ for $j \in \ZZ_+$.  Let $\theta \in \mathcal{D}(\R^d)$ be such that $\theta(0) = 1$ and consider its Fourier transform $\widehat{\theta}(\xi) =\int_{\R^d} \theta(x) e^{-2\pi i \xi \cdot x} dx$. Set $\rho_k(x) = k^d\widehat{\theta}(kx)$ for $k \in \ZZ_+$. We leave it to the reader to verify that $\varphi_{j,k} = ((\chi \ast \eta_j) \varphi) \ast \rho_k \in \mathcal{Z}^{(\omega)}_{(\V)}$ for all $j,k \in \ZZ_+$ and that for all $\varepsilon > 0$ there are $j_0,k_0 \in \ZZ_+$ such that $\| \varphi- \varphi_{j_0,k_0} \|_{\S^{\omega,h_{N}}_{1/v_{\lambda_{n+1}}}} \leq \varepsilon$. This shows the result. \\
	\noindent \emph{Roumieu case}: Choose $\chi \in \S^{\{\omega\}}_{\{\V\}}$ such that $\int_{\R^d}\chi(x)dx = 1$. Let $n_0, N_0 \in \N$ be such that $\chi \in \S^{\omega, h_{n_0}}_{v_{\lambda_{N_0}}}$.  It suffices to show that for all $n,N \in \N$ with $N \geq N_0 + 1$ and $n \geq n_0$
	$$
	\S^{\omega,h_{n}}_{1/v_{\lambda_{N+1}}} \subseteq \overline{ \Z^{\{\omega\}}_{\{\V\}}}^{\S^{\omega,h_{n+1}}_{1/v_{\lambda_{N}}}}.
	$$
Let $\varphi \in \S^{\omega,h_{n}}_{1/v_{\lambda_{N+1}}}$ be arbitrary. Choose $\eta \in \mathcal{D}(\R^d)$ such that $\eta(0) = 1$ and set $\eta_j(x) = \eta(x/j)$ for $j \in \ZZ_+$. We leave it to the reader to verify that $\varphi_j = (\chi \ast \eta_j)\varphi \in  \Z^{\{\omega\}}_{\{\V\}}$ for all $j \in \ZZ_+$ and that $\varphi_j \to \varphi$ in $\S^{\omega,h_{n+1}}_{1/v_{\lambda_{N}}}$ as $j \to \infty$.
\end{proof}

Next,  we discuss the mapping properties of the STFT on various function spaces related to the spectra  $(\Z^{\omega, h_{N}}_{(\V)})_{N \in \N}$ and $(\Z^{\{\omega\}}_{v_{\lambda_{N}}})_{N \in \N}$.
	
	\begin{lemma}\label{STFT-seq}\mbox{}
		
		\begin{itemize}
		\item[$(i)$] \emph{Beurling case}: For all  $m, n \in \N$  the mappings
					\begin{align*}
						&V_{\psi} : \S^{\omega, h_{m + 1}}_{1 / v_{\lambda_{n}}} \rightarrow C_{\frac{1}{v_{\lambda_{n + 1}}} \otimes e^{\frac{1}{h_{m}} \omega}}(\R^{2d}_{x, \xi}) \\
						&V_{\psi} : \S^{\omega, h_{m + 1}}_{v_{\lambda_{n+1}}} \rightarrow C_{v_{\lambda_n} \otimes e^{\frac{1}{h_{m}} \omega}}(\R^{2d}_{x, \xi}) \\
						&V_{\psi} : C_{v_{\lambda_{n+1}}} \rightarrow C_{v_{\lambda_n} \otimes 1}(\R^{2d}_{x, \xi}) 
					\end{align*}				
				are well-defined and continuous.
				
		\item[$(ii)$] \emph{Roumieu case}: For all  $m, n \in \N$, $n \geq 1$,  the mappings
					\begin{align*}
						&V_{\psi} : \S^{\omega, h_{m}}_{1 / v_{\lambda_{n + 1}}} \rightarrow C_{\frac{1}{v_{\lambda_{n}}} \otimes e^{\frac{1}{h_{m + 1}} \omega}}(\R^{2d}_{x, \xi}) \\
						&V_{\psi} : \S^{\omega, h_{m}}_{v_{\lambda_{n}}} \rightarrow C_{v_{\lambda_{n+1}} \otimes e^{\frac{1}{h_{m + 1}} \omega}}(\R^{2d}_{x, \xi}) \\
						&V_{\psi} : C_{v_{\lambda_{n}}} \rightarrow C_{v_{\lambda_{n+1}} \otimes 1} (\R^{2d}_{x, \xi})
					\end{align*}
				are well-defined and continuous.
				
		\end{itemize}

	\end{lemma}
	
	\begin{proof}
In view of the definition of the sequences $(h_{n})_{n \in \N}$ and $(\lambda_{n})_{n \in \N}$, this follows directly from Lemma \ref{l:STFTGS} and Lemma \ref{l:STFTC}.
	\end{proof}

We now give a variant of the Plancherel type identity \eqref{Plancherel}
	
	\begin{lemma}\label{lemma:STFTDesing}
		Let  $n \in \N$.
		\begin{itemize}
		\item[$(i)$]  \emph{Beurling case}: For all $f \in C_{v_{\lambda_{n + 3}}}$ and $\varphi \in \S^{\omega,1}_{1 / v_{\lambda_{n}}}$ it holds that
			\begin{equation}
				\label{eq:STFTDesing}
				\int_{\R^{d}} f(t) \varphi(t) dt = \int \int_{\R^{2d}} V_{\psi} f(x, \xi) V_{\overline{\psi}} \varphi(x, -\xi) dx d\xi.
			\end{equation}
			
\item[$(ii)$] \emph{Roumieu case}: \eqref{eq:STFTDesing} holds for all $f \in C_{v_{\lambda_{n}}}$ and $\varphi \in \Z^{\{\omega\}}_{1 / v_{\lambda_{n + 3}}}$. 
		\end{itemize}
	\end{lemma}
	
	\begin{proof}
	We only show the Roumieu case as the Beurling case is similar. By Lemma \ref{STFT-seq}$(ii)$ we have that $V_{\psi} f \in C_{v_{\lambda_{n + 1}} \otimes 1}(\R^{2d}_{x, \xi})$ and $V_{\psi} \varphi \in C_{1 / v_{\lambda_{n + 2}} \otimes e^{\omega / h_{m}}}(\R^{2d}_{x, \xi})$ for some $m \in \N$. Hence,
	\[ \int \int_{\R^{2d}} |V_{\psi} f(x, \xi) V_{\overline{\psi}} \varphi(x, -\xi)| dx d\xi < \infty .  \]
	Note that $fT_x \psi, \varphi T_x\psi \in L^2(\R^d)$ for all $x \in \R^d$ fixed. Since $V_\psi f(x,\xi) = \mathcal{F}(f T_x \overline{\psi}) (\xi)$ and  $V_{\overline{\psi}} \varphi(x, -\xi) = \overline{\mathcal{F}(\overline{\varphi T_x \psi})} (\xi)$, Plancherel's theorem and Fubini's theorem imply that
	\begin{align*}
	\int \int_{\R^{2d}} V_{\psi} f(x, \xi) V_{\overline{\psi}} \varphi(x, -\xi) dx d\xi 	&= \int_{\R^d}  \left(\int_{\R^{d}}  \mathcal{F}(f T_x \overline{\psi})(\xi) \overline{\mathcal{F}(\overline{\varphi T_x \psi})} (\xi)d\xi \right) dx \\
	&= \int_{\R^d}  \left(\int_{\R^{d}}  f(t)  \overline{\psi}(t-x) \varphi(t)\psi(t-x) dt \right) dx \\
	&= \int_{\R^{d}} f(t) \varphi(t) dt,
	\end{align*}
where we used that $\|\psi \|_{L^2} = 1$.
	\end{proof}
	

	In the following key lemma, we evaluate the condition from Theorem \ref{t:BarrelledImpliesP*2} for the spectra $(\Z^{\omega, h_{N}}_{(\V)})_{N \in \N}$ and $(\Z^{\{\omega\}}_{v_{\lambda_{N}}})_{N \in \N}$  in terms of the  short-time Fourier transform. We set
	$$
	C_{(\V)} = \varprojlim_{\lambda \to 0^+} C_{v_\lambda}.
	$$
	
	\begin{lemma} \mbox{}
		\label{l:wQSTFT} 
		\begin{itemize}
		
		\item[$(i)$] \emph{Beurling case}: If $\Z^{(\omega)}_{(\V)}$ is barrelled, then 
				\begin{gather*}
					\forall N \in \N ~ \exists M \geq N, n \in \N ~ \forall K \geq M, m \in \N ~ \exists k \in \N, C > 0~\forall f \in C_{(\V)} ~ : \\
					 \|V_{\psi} f\|_{v_{\lambda_{m}} \otimes e^{-\frac{1}{h_{M}} \omega}} \leq C \left( \|V_{\psi} f\|_{v_{\lambda_{n}} \otimes e^{-\frac{1}{h_{N}} \omega}} + \|V_{\psi} f\|_{v_{\lambda_{k}} \otimes e^{-\frac{1}{h_{K}} \omega}} \right) .
				\end{gather*}
				
		\item[$(ii)$] \emph{Roumieu case}: If $\Z^{\{\omega\}}_{\{\V\}}$ is barrelled, then
				\begin{gather*}
					\forall N \geq 2 ~ \exists M \geq N, n \in \N ~ \forall K \geq M, m \in \N ~ \exists k \in \N, C > 0~\forall f \in C_{v_{\lambda_{1}}} ~: \\
					 \|V_{\psi} f\|_{v_{\lambda_{M}} \otimes e^{-\frac{1}{h_{m}} \omega}} \leq C \left( \|V_{\psi} f\|_{v_{\lambda_{N}} \otimes e^{-\frac{1}{h_{n}} \omega}} + \|V_{\psi} f\|_{v_{\lambda_{K}} \otimes e^{-\frac{1}{h_{k}} \omega}} \right) .
				\end{gather*}
			\end{itemize}
	\end{lemma}
	
	\begin{proof}
		Again, we only show the Roumieu case as the Beurling case is similar. By Theorem \ref{t:BarrelledImpliesP*2} and Lemma \ref{l:ZStronglyReduced}, we have that
			\begin{gather*}
				\forall N \in \N ~ \exists M \geq N, n \in \N ~ \forall K \geq M, m \in \N ~ \exists k \in \N, C > 0~\forall f \in \left(\Z^{\{\omega\}}_{v_{\lambda_{N}}}\right)^{\prime}~:  \\
				 \|f\|^{*}_{\S^{\omega, h_{m}}_{1/v_{\lambda_{M}}}} \leq C \left( \|f\|^{*}_{\S^{\omega, h_{n}}_{1/v_{\lambda_{N}}}} + \|f\|^{*}_{\S^{\omega, h_{k}}_{1/v_{\lambda_{K}}}} \right) .
			\end{gather*}
		We may view any $f \in C_{v_{\lambda_{1}}}$ as an element of $\left(\Z^{\{\omega\}}_{v_{\lambda_{N}}}\right)^{\prime}$, $N \geq 2$, via
			\[ \ev{f}{\varphi} = \int_{\R^{d}} f(x) \varphi(x) dx , \qquad \varphi \in \Z^{\{\omega\}}_{v_{\lambda_{N}}} . \]
	Consequently, it suffices to show that 
			\begin{equation}
				\label{eq:STFTBoundDualNorm}
				\forall L \geq 2,  l \in \N ~ \exists C_{0} > 0 ~ \forall f \in C_{v_{\lambda_{1}}} ~:~ \|V_{\psi} f\|_{v_{\lambda_{L + 1}} \otimes e^{-\frac{1}{h_{l}} \omega}} \leq C_{0} \|f\|^{*}_{\S^{\omega, h_{l + 2}}_{1 / v_{\lambda_{L}}}} 
			\end{equation}
		and
			\begin{equation}
				\label{eq:DualNormBoundSTFT}
				\forall L \geq 2, l \in \N ~ \exists C_{1} > 0 ~ \forall f \in C_{v_{\lambda_{1}}} ~ : ~ \|f\|^{*}_{\S^{\omega, h_{l}}_{1 / v_{\lambda_{L + 2}}}} \leq C_{1} \|V_{\psi} f\|_{v_{\lambda_{L}} \otimes e^{-\frac{1}{h_{l+2}} \omega}}.
			\end{equation}
		We fist show \eqref{eq:STFTBoundDualNorm}. By Lemma \ref{l:TFSGS} and the fact that $h_{l+1} \geq Ah_l$, there is $C >0$ such that 
			\[ \|M_{\xi} T_{x} \psi\|_{\S^{\omega, h_{l + 2}}_{1 / v_{\lambda_{L}}}} \leq C \|\psi\|_{\S^{\omega, h_{l+1}}_{\widetilde{v}_{\lambda_{L}}}} \frac{1}{v_{\lambda_{L + 1}}(x)} e^{\frac{1}{h_{l}} \omega(\xi)}, \qquad (x,\xi) \in \R^{2d}. \]
		Set $C_{0} = C \|\psi\|_{\S^{\omega, h_{l+1}}_{\widetilde{v}_{\lambda_{L}}}}$. Then, for all $f \in C_{v_{\lambda_{1}}}$ and $(x,\xi) \in \R^{2d}$,
			\begin{align*} 
				|V_{\psi} f(x, \xi)| v_{\lambda_{L + 1}}(x) e^{-\frac{1}{h_{l}} \omega(\xi)} 
				&= |\ev{f}{ \overline{M_{\xi} T_{x}\psi}}| v_{\lambda_{L + 1}}(x) e^{-\frac{1}{h_{l}} \omega(\xi)}  \\
				&\leq \|f\|^{*}_{\S^{\omega, h_{l + 2}}_{1 / v_{\lambda_{L}}}} \|M_{\xi} T_{x} \psi\|_{\S^{\omega, h_{l + 2}}_{1 / v_{\lambda_{L}}}} v_{\lambda_{L + 1}}(x) e^{-\frac{1}{h_{l}} \omega(\xi)} \\
				&\leq C_{0} \|f\|^{*}_{\S^{\omega, h_{l+ 2}}_{1 / v_{\lambda_{L}}}}, 
			\end{align*}
		which shows the result. Next, we prove  \eqref{eq:DualNormBoundSTFT}. By  Lemma \ref{STFT-seq}$(ii)$, there is  $C > 0$ such that
		$$
		\|V_{\overline{\psi}} \varphi \|_{\frac{1}{v_{\lambda_{L + 1}}} \otimes e^{\frac{1}{ h_{l + 1}}\omega}} \leq C \| \varphi\|_{\S^{\omega, h_{l}}_{1 / v_{\lambda_{L + 2}}}}, \qquad \forall \varphi \in \S^{\omega, h_{l}}_{1 / v_{\lambda_{L + 2}}}.
		$$
		Hence, Lemma \ref{STFT-seq}$(ii)$ and Lemma \ref{lemma:STFTDesing}$(ii)$ imply that  for all  $f \in C_{v_{\lambda_{1}}}$ and $\varphi \in \S^{\omega, h_{l}}_{1 / v_{\lambda_{L + 2}}}$  it holds that 
			\begin{align*} 
				|\ev{f}{\varphi}| 
				&\leq \int \int_{\R^{2d}} |V_{\psi} f(x, \xi)| |V_{\overline{\psi}} \varphi(x, -\xi)| dx d\xi \\
				&\leq C \| \varphi\|_{\S^{\omega, h_{l}}_{1 / v_{\lambda_{L + 2}}}} \int \int_{\R^{2d}} |V_{\psi} f(x, \xi)| v_{\lambda_{L+1}}(x)e^{-\frac{1}{ h_{l + 1}}\omega(\xi)} dx d\xi \\
				&\leq C\|v_{\lambda_{L + 1}} / v_{\lambda_{L}}\|_{L^{1}} \|e^{-(\frac{1}{h_{l + 1}} - \frac{1}{h_{l+2}}) \omega}\|_{L^{1}} \|V_{\psi} f\|_{v_{\lambda_{L}} \otimes e^{-\frac{1}{h_{l + 2}} \omega}}  \| \varphi\|_{\S^{\omega, h_{l}}_{1 / v_{\lambda_{L + 2}}}} ,				
			\end{align*}
		which proves the result.
	\end{proof}
	
We are ready to show our main result.
	
	\begin{proof}[Proof of $(iii) \Rightarrow (i)$ in Theorem \ref{t:main}]
		Once more, we only show the Roumieu case as the  Beurling case is similar. By Lemma \ref{wQ}$(ii)$, it suffices to show that
			\begin{gather*}
				\forall N \geq 2 ~ \exists M \geq N, n \in \N ~ \forall K \geq M, m \in \N ~ \exists k \in \N, C > 0 ~ \forall (x, \xi) \in \R^{2d}~: \\
				v_{\lambda_{M}}(x) e^{-\frac{1}{h_{m}} \omega(\xi)} \leq C \left( v_{\lambda_{N}}(x) e^{-\frac{1}{h_{n}} \omega(\xi)} + v_{\lambda_{K}}(x) e^{-\frac{1}{h_{k}} \omega(\xi)} \right) .
			\end{gather*}
		Set $f_{x,\xi} = M_{\xi} T_{x} \psi \in C_{v_{\lambda_1}}$ for $(x,\xi) \in \R^{2d}$.
			Let $L \geq 2, l \in \N$ be arbitrary. The identity \eqref{repres} implies that for  all $(x,\xi) \in \R^{2d}$ 
			\begin{align*} 
				\|V_{\psi} (f_{x,\xi})\|_{v_{\lambda_{L}} \otimes e^{-\frac{1}{h_{l}} \omega}} 	&= \sup_{y, \eta \in \R^{d}} |V_{\psi} \psi(y - x, \eta - \xi)| v_{\lambda_{L}}(y) e^{-\frac{1}{h_{l}} \omega(\eta)} \\
				&\geq |V_{\psi} \psi(0, 0)| v_{\lambda_{L}}(x) e^{-\frac{1}{h_{l}} \omega(\xi)} \\
				&= v_{\lambda_{L}}(x) e^{-\frac{1}{h_{l}} \omega(\xi)},
			\end{align*}
			where we used that $\| \psi \|_{L^2} = 1$, and  (recall that $h_{n + 1} > S h_{n}$)
			\begin{align*}
				\|V_{\psi} (f_{x,\xi})\|_{v_{\lambda_{L+1}} \otimes e^{-\frac{1}{h_{l}} \omega}} &=  \sup_{y, \eta \in \R^{d}} |V_{\psi} \psi(y - x, \eta - \xi)| v_{\lambda_{L+1}}(y) e^{-\frac{1}{h_{l}} \omega(\eta)}\\
				&\leq C' e^{\frac{1}{h_{l}}} \|V_{\psi} \psi\|_{v_{\lambda_{L}} \otimes e^{\frac{1}{h_{l}} \omega}} v_{\lambda_{L}}(x) e^{-\frac{1}{h_{l + 1}} \omega(\xi)}.
			\end{align*}
				Hence, the result follows by applying the condition from Lemma \ref{l:wQSTFT}$(ii)$ to $f = f_{x,\xi}$, $(x,\xi) \in \R^{2d}$.
	\end{proof}

\end{document}